\documentclass[reqno,12pt]{amsart} 

\newtheorem{theorem}{Theorem}

\theoremstyle{remark}

\numberwithin{equation}{section}

\allowdisplaybreaks

\author{Michael J.\ Schlosser}
\address{Fakult\"at f\"ur Mathematik, Universit\"at Wien,
Oskar-Morgenstern-Platz~1, A-1090 Vienna, Austria}
\email{michael.schlosser@univie.ac.at}
\urladdr{http://www.mat.univie.ac.at/{\textasciitilde}schlosse}
\thanks{Partly supported by FWF Austrian Science Fund
grant F50-08.}

\title[$q$-Analogues of two product formulas of hypergeometric functions]
{$q$-Analogues of two product formulas of hypergeometric functions by Bailey}



\begin{document}

\begin{abstract}
We use Andrews' $q$-analogues of Watson's and Whipple's $_3F_2$
summation theorems to deduce two formulas for products of specific
basic hypergeometric functions. These constitute $q$-analogues of
corresponding product formulas for ordinary hypergeometric functions
given by Bailey. The first formula was obtained earlier by Jain and
Srivastava by a different method.
\end{abstract}

\dedicatory{Dedicated to Mourad E.H.\ Ismail}

\maketitle

\section{Introduction}\label{secintro}
We refer to Slater's text \cite{Sl} for an introduction to
hypergeometric series, and to Gasper and Rahman's text \cite{GR}
for an introduction to basic hypergeometric series, whose notations we follow.
Throughout, we assume $|q|<1$ and $|z|<1$.

In \cite{A}, George Andrews proved the following two theorems:
\begin{theorem}\label{qwatson}
\begin{equation}
{}_4\phi_3\!\left[\begin{matrix}a,b,c^{\frac 12},-c^{\frac 12}\\
(abq)^{\frac 12},-(abq)^{\frac 12},c
\end{matrix}\,;q,q\right]=a^{\frac n2}
\frac{(aq,bq,cq/a,cq/b;q^2)_\infty}
{(q,abq,cq,cq/ab;q^2)_\infty},
\end{equation}
where $b=q^{-n}$ and $n$ is a nonnegative integer.
\end{theorem}

\begin{theorem}\label{qwhipple}
\begin{equation}\label{qwhippleeq}
{}_4\phi_3\!\left[\begin{matrix}a,q/a,c^{\frac 12},-c^{\frac 12}\\
-q,e,cq/e\end{matrix}\,;q,q\right]=q^{\binom{n+1}2}
\frac{(ea,eq/a,caq/e,cq^2/ae;q^2)_\infty}
{(e,cq/e;q)_\infty},
\end{equation}
where $a=q^{-n}$ and $n$ is a nonnegative integer.
\end{theorem}
By a standard polynomial argument \eqref{qwhippleeq} also holds
when $a$ is a complex variable but $c=q^{-2n}$ with
$n$ being a nonnegative integer. (This is the case we will
make use of.)

Theorems~\ref{qwatson} and \ref{qwhipple} are $q$-analogues of Watson's
and of Whipple's $_3F_2$ summation theorems, listed as
Equations (III.23) and (III.24) in \cite[p.~245]{Sl}, respectively.

\section{Two product formulas for basic hypergeometric
functions}

We now have the following two product formulas which
are derived using Theorems~\ref{qwatson} and \ref{qwhipple}.
The first one in Theorem~3 was already given earlier by
Jain and Srivastava~\cite[Equation~(4.9)]{JS}
(as Slobodan Damjanovi\'c has kindly pointed out to the author,
after seeing an earlier version of this note),
who established the result by specializing a
general reduction formula for double basic hypergeometric series.
The second formula in Theorem~4 appears to be new.
\begin{theorem}\label{qwatsonprod}
\begin{equation}
{}_2\phi_1\!\left[\begin{matrix}a,-a\\
a^2\end{matrix}\,;q,z\right]
{}_2\phi_1\!\left[\begin{matrix}b,-b\\
b^2\end{matrix}\,;q,-z\right]
={}_4\phi_3\!\left[\begin{matrix}ab,-ab,abq,-abq\\
a^2q,b^2q,a^2b^2\end{matrix};q^2,z^2\right].
\end{equation}
\end{theorem}

\begin{theorem}\label{qwhippleprod}
\begin{subequations}\label{qwhippleprodid}
\begin{align}
{}_2\phi_1\!\left[\begin{matrix}a,q/a\\
-q\end{matrix}\,;q,z\right]
{}_2\phi_1\!\left[\begin{matrix}b,q/b\\
-q\end{matrix}\,;q,-z\right]
=\sum_{j=0}^\infty\frac{(q^{2-j}/ab,aq^{1-j}/b;q^2)_j}
{(q^2;q^2)_j}q^{\binom j2}(bz)^j&\\
={}_4\phi_3\!\left[\begin{matrix}ab,q^2/ab,aq/b,bq/a\\
-q^2,q,-q\end{matrix}\,;q^2,z^2\right]
\qquad\qquad\qquad\qquad\qquad\quad&\notag\\[.1em]
{}-\frac{(a-b)(1-q/ab)}{1-q^2}\,z\,
{}_4\phi_3\!\left[\begin{matrix}abq,q^3/ab,aq^2/b,bq^2/a\\
-q^2,q^3,-q^3\end{matrix}\,;q^2,z^2\right]&.
\end{align}
\end{subequations}
\end{theorem}

\begin{proof}[Sketch of proofs.]
To prove Theorem~\ref{qwatsonprod}, compare coefficients
of $z^n$. The resulting identity is equivalent to Theorem~\ref{qwatson}.
The proof of Theorem~\ref{qwhippleprod} is similar.
Comparison of coefficients of $z^n$ gives an identity which is
equivalent to Theorem~\ref{qwhipple} (where in the latter theorem
the restriction $a=q^{-n}$ is replaced by $c=q^{-2n}$, as mentioned).
The second identity in Equation \eqref{qwhippleprodid} follows from
splitting the sum over $j$ into two parts depending on the parity of $j$.
(This is motivated by the particular numerator factors in the $j$-th summand.)
The technical details -- elementary manipulation of
$q$-shifted factorials -- are routine and thus omitted.
\end{proof}

Theorem~\ref{qwatsonprod} is a $q$-analogue of Bailey's
formula in \cite[p.~246, Equation~(2.11)]{B}:
\begin{equation}\label{qwatsonprodido}
{}_1F_1\!\left[\begin{matrix}a\\
2a\end{matrix}\,;z\right]
{}_1F_1\!\left[\begin{matrix}b\\
2b\end{matrix}\,;-z\right]
={}_2F_3\!\left[\begin{matrix}\frac 12(a+b),\frac 12(a+b+1)\\
a+\frac 12,b+\frac 12,a+b\end{matrix};\frac 14z\right].
\end{equation}
To obtain \eqref{qwatsonprodido} from Theorem~\ref{qwatsonprod},
replace $(a,b,z)$ by $(q^a,q^b,(1-q)z/2)$, and let $q\to 1$.

Similarly, Theorem~\ref{qwhippleprod} is a $q$-analogue of Bailey's
formula in \cite[p.~245, Equation~(2.08)]{B}:
\begin{align}\label{qwhippleprodido}
{}_2F_0\!&\left[\begin{matrix}a,1-a\\
-\end{matrix}\,;z\right]
{}_2F_0\!\left[\begin{matrix}b,1-b\\
-\end{matrix}\,;-z\right]\notag\\[.2em]
&{}={}_4F_1\!\left[\begin{matrix}\frac 12(1+a-b),\frac 12(1-a+b),
\frac 12(a+b),\frac 12(2-a-b)\\[.1em]
\frac 12\end{matrix}\,;4z^2\right]\notag\\[.1em]
&\qquad{}-(a-b)(a+b-1)\,z\notag\\
&\qquad\times{}_4F_1\!\left[\begin{matrix}\frac 12(2+a-b),\frac 12(2-a+b),
\frac 12(1+a+b),\frac 12(3-a-b)\\[.1em]
\frac 32\end{matrix}\,;4z^2\right].
\end{align}
To obtain \eqref{qwhippleprodido} from Theorem~\ref{qwhippleprod},
replace $(a,b,z)$ by $(q^a,q^b,2z/(1-q))$ and let $q\to 1$.

\section{Related results in the literature}

A different product formula for basic hypergeometric functions
was established by Srivastava \cite[Eq.~(21)]{S1} (see also
\cite[Eq.~(3.13)]{S2}):
\begin{equation}
{}_2\phi_1\!\left[\begin{matrix}a,b\\
-ab\end{matrix}\,;q,z\right]
{}_2\phi_1\!\left[\begin{matrix}a,b\\
-ab\end{matrix}\,;q,-z\right]
={}_4\phi_3\!\left[\begin{matrix}a^2,b^2,ab,abq\\
a^2b^2,-ab,-abq\end{matrix};q^2,z^2\right].
\end{equation}
This formula is a $q$-extension of Bailey's formula in
\cite[p.~245, Equation~(2.08)]{B}
(or, equivalently, of an identity recorded by
Ramanujan~\cite[Ch.\ 13, Entry 24]{R}).

Finally, we mention that in 1941 F.H.~Jackson~\cite{J}
had derived the identity
\begin{equation}
{}_2\phi_1\!\left[\begin{matrix}a^2,b^2\\
a^2b^2q\end{matrix}\,;q^2,z\right]
{}_2\phi_1\!\left[\begin{matrix}a^2,b^2\\
a^2b^2q\end{matrix}\,;q^2,qz\right]
={}_4\phi_3\!\left[\begin{matrix}a^2,b^2,ab,-ab\\
a^2b^2,abq^{\frac 12},-abq^{\frac 12}\end{matrix};q,z\right],
\end{equation}
which is a $q$-analogue of Clausen's formula of 1828,
\begin{equation}
\left({}_2F_1\!\left[\begin{matrix}a,b\\
a+b+\frac 12\end{matrix}\,;z\right]\right)^2
={}_3F_2\!\left[\begin{matrix}2a,2b,a+b\\
2a+2b,a+b+\frac 12\end{matrix};z\right].
\end{equation}

Another $q$-analogue of Clausen's formula
was delivered by Gasper in \cite{G}.
While it has the advantage that it expresses a square of a
basic hypergeometric series as a  basic hypergeometric
series, it only holds provided the series terminate:
\begin{equation}\label{qClausen}
\left({}_4\phi_3\!\left[\begin{matrix}a,b,aby,ab/y\\
abq^{\frac 12},-abq^{\frac 12},-ab\end{matrix}\,;q,q\right]\right)^2
={}_5\phi_4\!\left[\begin{matrix}a^2,b^2,ab,aby,ab/y\\
a^2b^2,abq^{\frac 12},-abq^{\frac 12},-ab\end{matrix};q,q\right].
\end{equation}
See \cite[Sec.~8.8]{GR} for a nonterminating extension
of \eqref{qClausen} and related identities.

\section*{Acknowledgement}

I would like to thank George Gasper for his interest and for
informing me of the papers \cite{S1,S2} by Srivastava.
I am especially indebted to Slobodan Damjanovi\'c for pointing out that
Theorem~3 was already given by Jain and Srivastava~\cite[Equation~(4.9)]{JS}.

\end{document}